\documentclass[11pt,a4paper]{article}
\usepackage{amsthm,mathrsfs,amsmath,amssymb}
\usepackage{dsfont}
\usepackage{indentfirst}
\usepackage{cite}
\usepackage[colorlinks=true]{hyperref}
\hypersetup{urlcolor=blue, citecolor=red}
\numberwithin{equation}{section}
\allowdisplaybreaks
\date{}

\newtheorem{thm}{Theorem}[section]
\newtheorem{lem}[thm]{Lemma}

\newtheorem{xrem}[thm]{\bf Remark}
\newenvironment{prf}{\noindent {\bf Proof.}\rm}{\qed}
\newenvironment{prf1}{\noindent {\bf Proof of Theorem \ref{T1}.}\rm}{\qed}
\newenvironment{prf2}{\noindent {\bf Proof of Theorem \ref{T2}.}\rm}{\qed}
\newenvironment{prf3}{\noindent {\bf Proof of Theorem \ref{T3}.}\rm}{\qed}

\title{\bf Multiplicity and asymptotics of positive solutions for critical-concave Kirchhoff equation\footnote{Supported by National Natural Science Foundation of China (No.~12301143, No.~12261079, No.~12561019).}}

\author{{Zhi-Yun Tang$^a$, Gui-Dong Li$^b$, Yong-Yong Li$^c$\footnote{Corresponding author. E-mail address: mathliyy518@163.com}}\\
{\small\emph{$^a$School of Mathematics and Statistics, Southwest University, Chongqing, {\rm400715}, China,}}\\
{\small\emph{$^b$School of Mathematics and Statistics, Guizhou University, Guiyang, {\rm550025}, China,}}\\
{\small\emph{$^c$College of Mathematics and Statistics, Northwest Normal University, Lanzhou, {\rm730070}, China}}}

\begin{document}

\maketitle

\baselineskip 16.5pt

\setcounter{page}{1}

\begin{quote}
  {\bf Abstract:}\ This paper focuses on the critical Kirchhoff equation with concave perturbation
  \begin{align*}
  \begin{cases}
  \displaystyle  -\Big(a+b\int_\Omega|\nabla u|^2dx\Big)\Delta u=|u|^4u+\lambda|u|^{q-2}u\ \ &\mbox{in}\ \Omega,\\
  \displaystyle u=0\ \  &\mbox{on}\ \partial\Omega,
  \end{cases}
  \end{align*}
  where $\Omega$ is a smooth bounded domain in $\mathbb{R}^3$, $a,b,\lambda>0$ and $1<q<2$.
  By the constrained minimization methods, the mountain pass theorem and the concentration-compactness principle,
  we verify the multiplicity of positive solutions for $\lambda>0$ small enough.
  Moreover, we analyse the asymptotic behaviour of positive solutions as $b\rightarrow0$ and $\lambda\rightarrow0$, respectively.
  This work is a counterpart of [A. Ambrosetti et al., J.~Funct.~Anal. 1994] for the Kirchhoff equation.
  It is noteworthy that we don't require that $b>0$ is small enough here, which is imposed in the existing literatures to make refined estimates for the mountain pass level.

\end{quote}


\section{Introduction and main results}

In this paper, we are concerned with the existence and multiplicity of positive solutions to the following semilinear elliptic equation:
  \begin{align}\label{K2}
  \begin{cases}
  \displaystyle  -\Big(a+b\int_\Omega|\nabla u|^2dx\Big)\Delta u=f(x,u) &\mbox{in}\ \Omega,\\
  \displaystyle u=0 &\mbox{on}\ \partial\Omega,
  \end{cases}
  \end{align}
where $\Omega\subset\mathbb{R}^3$ is a smooth bounded domain, $a,b\geq 0$, $a+b>0$ and $f\in C(\Omega\times\mathbb{R},\mathbb{R})$. When $b>0$, Eq.~\eqref{K2} is called as Kirchhoff type equation.
In the physical context, Kirchhoff equation was proposed by Kirchhoff in \cite{K} to generalize the classical D'Alembert wave equations and study the steady state of  hyperbolic equation for the free vibrations of elastic strings. As we all know, Kirchhoff equation is of nonlocal type, since it takes into account the changes in length of the string produced by transverse vibrations. From the perspective of  variational method, the nonlocal term $\int_\Omega|\nabla u|^2dx$ in Eq.~\eqref{K2} may provoke some interesting mathematical difficulties, which draw many scholars' great attention after the Lion's pioneering work \cite{L}, in which an abstract functional analysis framework was established. As for Eq.~\eqref{K2} with $b>0$, there are some existence and nonexistence results on nontrivial solution presented in \cite{CXW,CKW,CT,CWL,K,L,LLG,LPKT,N,SL,XWT,ZL,ZT,ZF}.
Specially, D. Naimen \cite{N} studied the Sobolev critical Kirchhoff equation
  \begin{align*}
  \begin{cases}
  \displaystyle  -\Big(a+b\int_\Omega|\nabla u|^2dx\Big)\Delta u=|u|^4u+\lambda|u|^{p-2}u &\mbox{in}\ \Omega,\\
  \displaystyle u=0 &\mbox{on}\ \partial\Omega,
  \end{cases}
  \end{align*}
where $\Omega\subset\mathbb{R}^3$ is a bounded domain with smooth boundary, $a, \lambda \geq 0$ and $2\leq p<6$ (6 is the Sobolev critical exponent).
Therein, the nonexistence of nontrivial solution was proved if $\lambda=0$, $a\geq0$ and $\Omega$ is strictly star-shaped, nontrivial solution was obtained for any $\lambda>0$, $p\in(4,6)$ and $a\geq0$, meanwhile, positive solution was derived for any $a>0$, $p\in(2,4]$ and sufficiently large $\lambda>0$. Moreover, D.~Naimen proved a positive solution for any $p\in(2,4)$ and sufficiently small $\lambda>0$ if $a=0$, which was extended to the case of $p=4$ in \cite{ZL}, where multiple nontrivial solutions were obtained if $\lambda$ belongs to a small  left neighborhood of eigenvalue for the operator $-\int_\Omega|\nabla u|^2dx\Delta$. In \cite{ZT}, X.-J. Zhong and C.-L. Tang  studied the case of $p=2$ and $a=1$, they proved the existence of  multiple positive solutions and positive ground state solution when $\lambda$ falls in a small right neighborhood of the first eigenvalue of the operator $-\Delta$.

When $b=0$, Eq.~\eqref{K2} reduces to be a classical Schr$\ddot{\mbox{o}}$dinger equation.
For the Sobolev critical case we concerned here, there are three celebrated works \cite{ABC,BN,P}.
As a starting point, S.I. Poho$\check{\mbox{z}}$aev \cite{P} showed that if $\Omega\subset\mathbb{R}^3$ is strictly star-shaped there is no solution for the equation
\begin{align*}
\begin{cases}
\displaystyle -\Delta u=u^5 &\mbox{in}\ \Omega,\\
\displaystyle u>0 &\mbox{in}\ \Omega, \\
\displaystyle u=0 &\mbox{on}\ \partial\Omega.
\end{cases}
\end{align*}
Later, H. Br$\acute{\mbox{e}}$zis and L. Nirenberg reversed the above well-known nonexistence result by adding lower-order terms in \cite{BN}.
Concretely, the linear perturbed term $\lambda u$ was firstly added to $u^5$, when $\Omega$ is a ball they proved that there exists solution iff $\lambda\in(\frac14\lambda_1,\lambda_1)$, where $\lambda_1$ is the first eigenvalue of $-\Delta$. Further, the superlinear perturbed term $\mu u^s$ $(1<s<5\ \mbox{and}\ \mu>0)$ was added to $u^5$, they proved that there exists solution for any $\mu>0$ if $3<s<5$, or for $\mu>0$ large enough if $1<s\leq 3$. After this, A. Ambrosetti, H. Br$\acute{\mbox{e}}$zis and G. Cerami \cite{ABC} added sublinear (concave) perturbed term, namely, they studied the Schr$\ddot{\mbox{o}}$dinger equation
\begin{align}\label{K11}
\begin{cases}
\displaystyle -\Delta u=u^{p-1}+\lambda u^{q-1} &\mbox{in}\ \Omega, \\
\displaystyle u>0 &\mbox{in}\ \Omega, \\
\displaystyle u=0 &\mbox{on}\ \partial\Omega,
\end{cases}
\end{align}
where $\Omega\subset\mathbb{R}^3$ is a bounded domain with smooth boundary $\partial\Omega$, $\lambda>0$ and $1<q<2<p\leq 6$, such nonlinearity is called as concave-convex nonlinearity if $p<6$ and critical-concave nonlinearity if $p=6$. Based on the sub- and supersolution method, they proved that Eq.~\eqref{K11} has solution iff $\lambda\in(0,\lambda_*]$ for some $\lambda_*>0$. Indeed, for $\lambda>0$ sufficiently small, observe that the energy functional of Eq.~\eqref{K11} has at least two positive critical points in the direction of any nontrivial $u$ under the fibering map $tu$, then it is natural to look for two positive solutions for Eq.~\eqref{K11}. Such expected result was confirmed in \cite{ABC}. Precisely, for any $\lambda\in(0,\lambda_*)$, Ambrosetti et al.~proved that Eq.~\eqref{K11} has a first solution $u_\lambda$ which is a local minimizer with negative energy. Having truncated the concave-convex nonlinearity, by applying the mountain pass theorem they further obtained a second solution $v_\lambda$ for any $\lambda\in(0,\lambda_*)$. Moreover, they showed that, in the sense of subsequence, $|u_\lambda|_\infty\rightarrow0$ as $\lambda\to0$, while $|v_\lambda|_\infty\to+\infty$ as $\lambda\to0$ if $\Omega$ is strictly star-shaped. Afterwards, many researchers tried their best to extend those results in \cite{ABC} to various elliptic equations. As for Kirchhoff equations, we refer readers to \cite{CXW,CWL,CKW,LLG,LPKT,SL}. Particularly, when it comes to the critical case of Eq.~\eqref{K11} studied in \cite{ABC}, a counterpart for Kirchhoff equation is
\begin{align}\label{K1}
\begin{cases}
\displaystyle  -\Big(a+b\int_\Omega|\nabla u|^2dx\Big)\Delta u=|u|^4u+\lambda|u|^{q-2}u &\mbox{in}\ \Omega,\\
\displaystyle u=0 &\mbox{on}\ \partial\Omega,
\end{cases}
\end{align}
where $\Omega\subset \mathbb{R}^3$ is a bounded domain with smooth boundary, the parameters $a,b,\lambda>0$ and $1<q<2$.
Similar to the existence result in \cite{ABC}, an open question left in many scholars' mind is
\begin{enumerate}
  \item [$\mathbf{(Q)}$] whether Eq.~\eqref{K1} has two positive solutions for $\lambda>0$ sufficiently small.
\end{enumerate}

Firstly, Y.J. Sun and X. Liu \cite{SL} split the Nehari manifold $\mathcal{N}_\lambda$ of Eq.~\eqref{K1} into three disjoint parts $\mathcal{N}_\lambda^+, \mathcal{N}_\lambda^0, \mathcal{N}_\lambda^-$ and verified $\mathcal{N}_\lambda^\pm\neq\emptyset$, $\mathcal{N}_\lambda^0=\{0\}$ for $\lambda>0$ small~enough. By using the Ekeland variational principle and the concentration-compactness principle in $\mathcal{N}_\lambda^+$, they proved that Eq.~\eqref{K1} admits one positive solution with negative energy for $\lambda>0$ small enough and $a,b>0$.
As a conjecture, it is possible to search for a second positive solution in $\mathcal{N}_\lambda^-$ by using the similar way.
However, as they explained, it is hard to obtain the convergence of minimizing sequence in $\mathcal{N}_\lambda^-$,
and the existence of a second positive solution cannot be proved in their way directly.

Later, C.-Y. Lei, G.-S. Liu and  L.-T. Guo \cite{LLG} further studied the existence of two positive solutions for Eq.~\eqref{K1} with $\lambda$, $b$ small enough.
By the constrained minimization argument, they obtained one positive solution $u_1$ (which is a local minimizer) of Eq.~\eqref{K1} for $\lambda>0$ small enough.
Then, having established a mountain pass geometry originating from the first positive solution $u_1$, they obtained a second positive solution for $b,\lambda>0$ small enough.
However, the methods in \cite{LLG} are invalid to obtain a second positive solution for Eq.~\eqref{K1} without the small restriction of $b>0$.

Until now, the question $\mathbf{(Q)}$ remains open.
We have long investigated this problem and carried out related work.
Recently, motivated by the interesting work \cite{CT}, we achieve new progress in our study.
We derived a new threshold value below which the functional satisfies the PS condition,
thus solving this problem completely (see Theorem \ref{T1}).
To our surprise, the solution corresponding to the local minimum identified in Theorem \ref{T1} is the ground state solution---an outcome that is both novel and interesting, even when considering the specific case of $b=0$.

Now we state our main results as follows.

\begin{thm}\label{T1}
Assume that $\Omega\subset \mathbb{R}^3$ is a bounded domain with smooth boundary, $a>0$ and $1<q<2$.
Then there exists some $\Lambda>0$ (independent of $b$) such that Eq.~\eqref{K1} has at least two positive solutions $u_{b,\lambda}^1$ and $u_{b,\lambda}^2$ for any $b>0$ and $\lambda\in(0,\Lambda)$, where $u_{b,\lambda}^1$ is a positive ground state solution with negative energy and $u_{b,\lambda}^2$ is a positive mountain pass solution with positive energy.
\end{thm}

\begin{xrem}\label{Re1}
This theorem is the counterpart of the multiplicity result in \cite{ABC} with respect to Kirchhoff equation and gives a complete positive answer to the question $\mathbf{(Q)}$. Comparing to \cite{LLG}, we don't need the small restriction of $b>0$ here, some new tricks will be introduced.
An interesting by-product is that the positive ground state solution $u_{b,\lambda}^1$ is a local minimizer of the energy functional for Eq.~\eqref{K1}.
It is worth mentioning that our approaches are still applicable to dealing with the degenerate case (i.e. $a=0$ and $b>0$) of Eq.~\eqref{K1}, of which multiple positive solutions can be obtained by the same strategy. Furthermore, once $a=0$ and $b>0$, the restriction on $q$ can be extended from $1<q<2$ to $1<q<4$.
\end{xrem}

Further, we study the asymptotic behaviour of positive solutions as $b\rightarrow0$ and $\lambda\rightarrow0$, respectively.

\begin{thm}\label{T2}
Assume that $\Omega\subset \mathbb{R}^3$ is a bounded domain with smooth boundary, $a>0$, $\lambda\in(0,\Lambda)$~and $1<q<2$.
Then, for any sequence $\{b_n\}\subset(0,+\infty)$ with $b_n\rightarrow0$, $u_{b_n,\lambda}^1\xrightarrow{n}u_\lambda^1$ and $u_{b_n,\lambda}^2\xrightarrow{n} u_\lambda^2$ in $H_0^1(\Omega)$ in the sense of subsequence, where $u_\lambda^1$ (with negative energy) and $u_\lambda^2$ (with positive energy) are two positive solutions of the Schr$\ddot{\mbox{o}}$dinger equation
\begin{align}\label{SC1}
\begin{cases}
\displaystyle  -a\Delta u=|u|^4u+\lambda|u|^{q-2}u &\mbox{in}\ \Omega,\\
\displaystyle u=0 &\mbox{on}\ \partial\Omega.
\end{cases}
\end{align}
\end{thm}

\begin{xrem}\label{Re4}
Obviously, this theorem indicates that the multiplicity of positive solutions to Eq.~\eqref{K1} cannot vanish as the parameter $b\rightarrow0^+$.
Moreover, as a by-product, we obtain two positive solutions of Eq.~\eqref{SC1} for $\lambda>0$ sufficiently small, which consist with the multiplicity result in \cite{ABC} to some extent.
\end{xrem}

\begin{thm}\label{T3}
Assume $\Omega\subset \mathbb{R}^3$ is a bounded domain with smooth boundary, $a, b>0$ and $1<q<2$.
Then, for any sequence $\{\lambda_n\}\subset(0,\Lambda)$ with $\lambda_n\rightarrow0$, the solution sequence $u_{b,\lambda_n}^1\xrightarrow{n}0$ in $H_0^1(\Omega)$, while $|u_{b,\lambda_n}^2|_\infty\xrightarrow{n}+\infty$ if $\Omega$ is strictly star-shaped.
\end{thm}

\begin{xrem}\label{Re3}
The assumption that $\Omega$ is strictly star-shaped is imposed to conclude the nonexistence result of the limiting problem (i.e.~Eq.~\eqref{K1} with $\lambda=0$),
by which we analyse the asymptotic behavior of the mountain pass solution sequence~$\{u_{b,\lambda_n}^2\}$.
Analogously, based on the existence results on Eq.~\eqref{SC1} obtained in Theorem \ref{T2}, let $\Omega$ be strictly star-shaped if necessary, the vanishing property of $u_\lambda^1$ and the blow-up property of $u_\lambda^2$ as $\lambda\rightarrow0^+$ can be proved by using the same methods as the proof of Theorem~\ref{T3}.
\end{xrem}

The rest of this paper are organized as follows: Theorem \ref{T1} will be proved in Sect.~\ref{Se1}, Theorems \ref{T2} and \ref{T3} will be proved in Sect.~\ref{Se2}.
Throughout this paper, we will frequently use the notations:
\begin{enumerate}
  \setlength{\itemsep}{-0.01pt}
  \item [$\spadesuit$] $C_0^\infty(\Omega)$ consists of infinitely differentiable functions with compact support in $\Omega$.
  \item [$\spadesuit$] $L^p(\Omega)$ denotes the Lebesgue space with the usual norm $|u|_p=\left(\int_\Omega|u|^pdx\right)^{\frac1p}$ for $p\in[1,+\infty)$.
  \item [$\spadesuit$] $L^\infty(\Omega)$ is the space of measurable functions with the norm $|u|_\infty=\mbox{ess~sup}_{x\in\Omega}|u(x)|$.
  \item [$\spadesuit$] $H_0^1(\Omega)$ is the usual Sobolev space endowed with the standard norm $\|u\|=\left(\int_\Omega|\nabla u|^2dx\right)^{\frac12}$.
  \item [$\spadesuit$] The best Sobolev constant $S=\inf\left\{\|u\|^2:u\in H_0^1(\Omega)\ \mbox{and}\ |u|_6=1\right\}$.
  \item [$\spadesuit$] $\mathbb{B}_r(y):=\left\{x\in\Omega:|x-y|<r\right\}$ and $\overline{\mathbb{B}}_r(y):=\left\{x\in\Omega:|x-y|\leq r\right\}$.
  \item [$\spadesuit$] $|\Omega|$ denotes the Lebesgue measure of $\Omega$ and $u^\pm=\max\left\{\pm u,0\right\}$.
  \item [$\spadesuit$] $o(1)$ denotes a quantity tending to $0$ as $n\rightarrow\infty$.
\end{enumerate}

\section{Proof of Theorem \ref{T1}}\label{Se1}

From the perspective of variational methods, solutions of Eq.~\eqref{K1} correspond to critical points of the energy functional
\begin{align*}
\mathcal{I}_{b,\lambda}(u)=\frac a2\|u\|^2+\frac b4\|u\|^4-\frac16|u|_6^6-\frac\lambda q|u|_q^q,\ \ \ \ \forall\ u\in H_0^1(\Omega).
\end{align*}
It is standard to verify $\mathcal{I}_{b,\lambda}\in C^1(H_0^1(\Omega),\mathbb{R})$.
To estimate energy, we introduce the auxiliary functional
\begin{align*}
\mathcal{J}_{b,\lambda}(u)
=\mathcal{I}_{b,\lambda}(u)+\frac1{24}\left[b^2S^4+4\left(a+b\|u\|^2\right)S\right]^{\frac32}
+\frac{b^2S^3}4\|u\|^2-\frac{\left(b^2S^4+4aS\right)^{\frac32}}{24}.
\end{align*}
Similarly, it is easy to verify that $\mathcal{J}_{b,\lambda}\in C^1(H_0^1(\Omega),\mathbb{R})$ and for any $u,v\in H_0^1(\Omega)$,
\begin{align*}
\left\langle\mathcal{J}_{b,\lambda}'(u),v\right\rangle
&=\left[a+\frac{b^2S^3}2+b\|u\|^2+\frac{bS}2\sqrt{b^2S^4+4(a+b\|u\|^2)S}\right]\int_\Omega\nabla u\cdot\nabla v dx \\
&\quad-\int_\Omega|u|^4uvdx-\lambda \int_\Omega |u|^{q-2}uvdx.
\end{align*}

\subsection{Existence of positive solution with negative energy}

In this section, we show that Eq.~\eqref{K1} has a positive solution $u_{b,\lambda}^0$, which is local  minimizer of $\mathcal{I}_{b,\lambda}$,
and $\mathcal{J}_{b,\lambda}$ has a local positive minimizer $w_{b,\lambda}^1$ for small $\lambda$, where $\mathcal{I}_{b,\lambda}(u_{b,\lambda}^0)<0$ and $\mathcal{J}_{b,\lambda}(w_{b,\lambda}^1)<0$.

\begin{lem}\label{L2.1}
There exist some $\lambda_1>0$ and $\rho,\alpha>0$ (independent of $b$) such that, for any $b>0$ and $\lambda\in(0,\lambda_1)$, there hold
\begin{align}
&\inf\limits_{u\in S_\rho}\mathcal{I}_{b,\lambda}(u)\geq\alpha>0\ \ \mbox{and}\ \ m_{b,\lambda}:=\inf\limits_{u\in B_\rho(0)}\mathcal{I}_{b,\lambda}(u)<0, \label{R2.1.1}\\
&\inf\limits_{u\in S_\rho}\mathcal{J}_{b,\lambda}(u)\geq\alpha>0\ \ \mbox{and}\ \ \hat{m}_{b,\lambda}:=\inf\limits_{u\in B_\rho(0)}\mathcal{J}_{b,\lambda}(u)<0, \label{R2.1.2}
\end{align}
where $S_\rho=\left\{u\in H_0^1(\Omega):\|u\|=\rho\right\}$ and $B_\rho(0)=\left\{u\in H_0^1(\Omega):\|u\|\leq\rho\right\}$.
\end{lem}

\begin{prf}
On the one hand, by the H$\ddot{\mbox{o}}$lder and Sobolev inequalities we have, for any $b>0$ and $\lambda>0$,
\begin{align*}
\mathcal{J}_{b,\lambda}(u)\geq
\mathcal{I}_{b,\lambda}(u)
&\geq\frac a2\|u\|^2+\frac b4\|u\|^4-\frac1{6S^3}\|u\|^6-\frac{\lambda|\Omega|^{\frac{6-q}6}}{qS^{\frac q2}}\|u\|^q\\
&\geq\|u\|^q\left(\frac a2\|u\|^{2-q}-\frac1{6S^3}\|u\|^{6-q}-\frac{\lambda|\Omega|^{\frac{6-q}6}}{qS^{\frac q2}}\right),\ \ \ \forall\ u\in H^1_0(\Omega).
\end{align*}
Set
\begin{align*}
f(t)=\frac a2t^{2-q}-\frac1{6S^3}t^{6-q},\ \ \ \forall\ t\geq 0,
\end{align*}
it is easy to show that $\rho=\left[\frac{3a(2-q)S^3}{6-q}\right]^{\frac14}$ is the unique maximum point of $f$ on $[0,+\infty)$ and the~maximum
\begin{align*}
f(\rho)=\frac {2a}{6-q}\left[\frac{3a(2-q)S^3}{6-q}\right]^{\frac{2-q}4}.
\end{align*}
Then, arbitrarily extracting $\lambda_1\in\left(0,f(\rho)qS^{\frac q2}|\Omega|^{\frac{q-6}6}\right)$, we deduce that, for any $b>0$ and $\lambda\in(0,\lambda_1)$,
\begin{align*}
\inf\limits_{u\in S_\rho}\mathcal{J}_{b,\lambda}(u)\geq
\inf\limits_{u\in S_\rho}\mathcal{I}_{b,\lambda}(u)\geq
\rho^q\left(f(\rho)-\frac{\lambda_1|\Omega|^{\frac{6-q}6}}{qS^{\frac q2}}\right):=\alpha>0.
\end{align*}
On the other hand, since $(1+s)^{\frac32}\leq 1+2s$ for small $s\geq0$,
we deduce, for any given $u\in H_0^1(\Omega)\backslash\{0\}$,
\begin{align*}
\lim\limits_{t\rightarrow0^+}\frac{\mathcal{I}_{b,\lambda}(tu)}{t^q}
&\leq\lim\limits_{t\rightarrow0^+}\frac{\mathcal{J}_{b,\lambda}(tu)}{t^q} \notag\\
&\leq \lim\limits_{t\rightarrow0^+}
\left[\left(\frac a2+\frac{b^2S^3}4+\frac{bS}3\sqrt{b^2S^4+4aS}\right)t^{2-q}\|u\|^2+\frac {bt^{4-q}}4\|u\|^4-\frac{t^{6-q}}6|u|_6^6-\frac{\lambda} q|u|_q^q\right] \notag\\
&=-\frac\lambda q|u|_q^q<0,
\end{align*}
which implies $m_{b,\lambda}, \hat{m}_{b,\lambda} \in(-\infty,0)$ for any $b,  \lambda>0$. Thus the proof of this lemma is completed.
\end{prf}

\begin{thm}\label{T2.2}
For any $b>0$ and $\lambda\in(0,\lambda_1)$, Eq.~$\mathrm{(\ref{K1})}$ has a positive solution $u_{b,\lambda}^1$ with the energy $\mathcal{I}_{b,\lambda}(u_{b,\lambda}^1)=m_{b,\lambda}$.
\end{thm}

\begin{prf}
As a by-product of proving Lemma \ref{L2.1}, there holds that
\begin{align}\label{R2.21}
\frac a2\|u\|^2+\frac b4\|u\|^4-\frac16|u|_6^6\geq0,\ \ \ \ \forall\ u\in B_\rho(0).
\end{align}
Since $\mathcal{I}_{b,\lambda}$ is even, by the definition of $m_{b,\lambda}$, there exists nonnegative minimizing sequence $\{u_n\}\subset B_\rho(0)$ such that $\mathcal{I}_{b,\lambda}(u_n)\xrightarrow{n}m_{b,\lambda}$. Clearly, $\{u_n\}$ is bounded in $H_0^1(\Omega)$.
Then there exists some nonnegative $u_{b,\lambda}^1$ such that, up to a subsequence,
\begin{align}\label{R2.22}
u_n\rightharpoonup u_{b,\lambda}^1\ \mbox{in}\ H_0^1(\Omega);\ \ \ \
u_n\rightarrow u_{b,\lambda}^1\ \mbox{in}\ L^p(\Omega),\ 1\leq p<6;\ \ \ \
u_n(x)\rightarrow u_{b,\lambda}^1(x)\ \mbox{a.e.}\ \mbox{in}\ \Omega.
\end{align}
Set $\bar{u}_n=u_n-u_{b,\lambda}^1$. By \eqref{R2.22}, we have
\begin{align}\label{R2.23}
\|u_n\|^2=\|\bar{u}_n\|^2+\|u_{b,\lambda}^1\|^2+o(1).
\end{align}
We claim $\bar{u}_n\in B_\rho(0)$ for $n$ large. Indeed, if $u_{b,\lambda}^1=0$, it is easy to see $\bar{u}_n=u_n\in B_\rho(0)$ for all $n$;
if $u_{b,\lambda}^1\neq0$, \eqref{R2.23} implies $\|\bar{u}_n\|<\rho$ for large $n$. Thus our claim is true.
Observe $\|u_{b,\lambda}^1\|\leq\liminf\limits_{n\rightarrow\infty}\|u_n\|\leq\rho$, then it follows from \eqref{R2.21}, \eqref{R2.22} and the Br$\acute{\mbox{e}}$zis-Lieb lemma (see \cite[Lemma 1.32]{W}) that
\begin{align*}
m_{b,\lambda}
&=\lim_{n\to\infty}\mathcal{I}_{b,\lambda}(u_n) \\
&=\lim_{n\to\infty}\left(\mathcal{I}_{b,\lambda}(u_{b,\lambda}^1)+\frac a2\|\bar{u}_n\|^2+\frac b4\|\bar{u}_n\|^4+\frac b2\|\bar{u}_n\|^2\|u_{b,\lambda}^1\|^2-\frac16|\bar{u}_n|_6^6\right) \\
&\geq \mathcal{I}_{b,\lambda}(u_{b,\lambda}^1) \\
&\geq m_{b,\lambda}.
\end{align*}
Hence, $\mathcal{I}_{b,\lambda}(u_{b,\lambda}^1)=m_{b,\lambda}$.
Namely, $u_{b,\lambda}^1\neq0$ is nonnegative local minimizer of $\mathcal{I}_{b,\lambda}$.
Clearly, $\|u_{b,\lambda}^1\|<\rho$.

We next show $\mathcal{I}_{b,\lambda}'(u_{b,\lambda}^1)=0$.
For any $\varphi\in H_0^1(\Omega)$, let $t>0$ be small such that $u_{b,\lambda}^1+t\varphi\in B_\rho(0)$, then there results
\begin{align*}
0
&\leq \mathcal{I}_{b,\lambda}(u_{b,\lambda}^1+t\varphi)-\mathcal{I}_{b,\lambda}(u_{b,\lambda}^1) \notag\\
&=\left(\|u_{b,\lambda}^1+t\varphi\|^2-\|u_{b,\lambda}^1\|^2\right)
\left[\frac a2+\frac b4\left(\|u_{b,\lambda}^1+t\varphi\|^2+\|u_{b,\lambda}^1\|^2\right)\right] \notag\\
&\quad-\frac16\left(|u_{b,\lambda}^1+t\varphi|_6^6-|u_{b,\lambda}^1|_6^6\right)
-\frac\lambda q\left(|u_{b,\lambda}^1+t\varphi|_q^q-|u_{b,\lambda}^1|_q^q\right).
\end{align*}
Dividing both sides of the above inequality by $t$ and letting $t\to0^+$, we derive that, for any $\varphi\in H_0^1(\Omega)$,
\begin{align}\label{R2.25}
0\leq\left(a+b\|u_{b,\lambda}^1\|^2\right)\int_\Omega\nabla u_{b,\lambda}^1\cdot\nabla \varphi dx
-\int_\Omega (u_{b,\lambda}^1)^5\varphi dx
-\lambda\int_\Omega (u_{b,\lambda}^1)^{q-1}\varphi dx.
\end{align}
If replacing $\varphi$ with $-\varphi$ in \eqref{R2.25}, there results that, for any $\varphi\in H_0^1(\Omega)$,
\begin{align*}
\left(a+b\|u_{b,\lambda}^1\|^2\right)\int_\Omega\nabla u_{b,\lambda}^1\cdot\nabla \varphi dx
-\int_\Omega (u_{b,\lambda}^1)^5\varphi dx
-\lambda\int_\Omega (u_{b,\lambda}^1)^{q-1}\varphi dx=0.
\end{align*}
That is, $u_{b,\lambda}^1$ is indeed a nonnegative nontrivial solution of Eq.~\eqref{K1}.
As a consequence of the strong maximum principle of weak solutions, $u_{b,\lambda}^1(\cdot)>0$ in $\Omega$.
Thus we complete the proof of this theorem.
\end{prf}

\vspace{6.00pt}
We may prove the forthcoming lemma and omit the similar details to the proof of Theorem \ref{T2.2}.

\begin{lem}\label{L2.3}
For any $b>0$ and $\lambda\in(0,\lambda_1)$, $\mathcal{J}_{b,\lambda}$ has a local positive minimizer $w_{b,\lambda}^1\in C^2(\Omega)$ with $\mathcal{J}_{b,\lambda}(w_{b,\lambda}^1)=\hat{m}_{b,\lambda}$ and $\mathcal{J}_{b,\lambda}'(w_{b,\lambda}^1)=0$ in $H_0^{-1}(\Omega)$.
\end{lem}

\begin{prf}
Repeating the arguments in proving Theorem \ref{T2.2}, we conclude that there exists some positive $w_{b,\lambda}^1\in H_0^1(\Omega)$ such that $\|w_{b,\lambda}^1\|<\rho$,  $\mathcal{J}_{b,\lambda}(w_{b,\lambda}^1)=\hat{m}_{b,\lambda}$ and
\begin{align*}
-\Delta w_{b,\lambda}^1
&=\frac{\big[(w_{b,\lambda}^1)^5+\lambda (w_{b,\lambda}^1)^{q-1}\big](1+w_{b,\lambda}^1)}
{\left[a+\frac{b^2S^3}2+b\|w_{b,\lambda}^1\|^2+\frac{bS}2\sqrt{b^2S^4+4\big(a+b\|w_{b,\lambda}^1\|^2\big)S} \right]\big(1+w_{b,\lambda}^1(x)\big)}
\ \ \ \mbox{in}\ H_0^{-1}(\Omega)\\
&:=h(w_{b,\lambda}^1(x))\big(1+w_{b,\lambda}^1\big).
\end{align*}
Owing to
$$
0\leq h(w_{b,\lambda}^1(x))\leq\frac{1+\lambda_1}a\left[1+\big(w_{b,\lambda}^1(x)\big)^4\right],\ \ \ \forall\ x\in\Omega,
$$
which implies $h\circ w_{b,\lambda}^1\in L^{\frac32}(\Omega)$. Then we deduce from the Brezis-Kato estimate that $w_{b,\lambda}^1\in L^p(\Omega)$ and from the $L^p$-theory of elliptic equation that $w_{b,\lambda}^1\in W^{2,p}(\Omega)$ for any $p\in[1,+\infty)$. Further, it follows from Morrey's embedding theorem that $w_{b,\lambda}^1\in C^{1,\alpha}(\overline{\Omega})$ and from the Schauder estimate that $w_{b,\lambda}^1\in C^{2,\alpha}(\overline{\Omega})$ for any $\alpha\in(0,1)$. That is, $w_{b,\lambda}^1\in C^2(\Omega)$. Thus the proof of this lemma is finished.
\end{prf}

\vspace{8pt}
Further, we prove that $\hat{m}_{b,\lambda}$ equals to the least energy of all critical points of $\mathcal{J}_{b,\lambda}$ for $\lambda>0$ small enough and any $b>0$.
\begin{lem}\label{L2.3+}
There exists some $\lambda_2\in\left(0,\lambda_1\right]$ (independent of $b$) such that
$\hat{m}_{b,\lambda}=\tilde{m}_{b,\lambda}:=\inf\limits_{\widetilde{\mathcal{N}}_{b,\lambda}}\mathcal{J}_{b,\lambda}$ for any $b>0$ and $\lambda\in(0,\lambda_2)$, where
\begin{align*}
\widetilde{\mathcal{N}}_{b,\lambda}=\left\{u\in H_0^1(\Omega)\backslash\{0\}:\mathcal{J}_{b,\lambda}'(u)=0\ \mbox{in}\ H_0^{-1}(\Omega)\right\}.
\end{align*}
\end{lem}

\begin{prf}
We first verify that $\tilde{m}_{b,\lambda}$ can be achieved.
There exists a sequence $\{u_n\}\subset \widetilde{\mathcal{N}}_{b,\lambda}$ such that
\begin{align*}
\mathcal{J}_{b,\lambda}(u_n)\xrightarrow{n}\tilde{m}_{b,\lambda}\ \ \ \ \mbox{and}\ \ \ \ \mathcal{J}_{b,\lambda}'(u_n)=0\ \mbox{in}\ H_0^{-1}(\Omega).
\end{align*}
By the H$\ddot{\mbox{o}}$lder and Sobolev inequalities, there results that
\begin{align}\label{R2.3.1}
\tilde{m}_{b,\lambda}+o(1)
&=\mathcal{J}_{b,\lambda}(u_n)-\frac14\left\langle\mathcal{J}_{b,\lambda}'(u_n),u_n\right\rangle \notag\\
&=\frac a4\|u_n\|^2+\frac1{12}|u_n|_6^6-\frac{\lambda(4-q)}{4q}|u_n|_q^q
+\frac{bS}{24}\left[b^2S^4+4(a+b\|u_n\|^2)S\right]^{\frac12}\|u_n\|^2 \notag\\
&\quad+\frac1{24}\left[b^2S^4+4(a+b\|u_n\|^2)S\right]^{\frac12}(b^2S^4+4aS)-\frac1{24}(b^2S^4+4aS)^{\frac32} \notag\\
&\geq \frac a4\|u_n\|^2-\frac{\lambda(4-q)|\Omega|^{\frac{6-q}6}}{4qS^{\frac q2}}\|u_n\|^q.
\end{align}
It follows from Lemma \ref{L2.3} that $\tilde{m}_{b,\lambda}\leq\mathcal{J}_{b,\lambda}(w_{b,\lambda}^1)=\hat{m}_{b,\lambda}<0$. Setting
\begin{align*}
\lambda_2=\min\left\{\lambda_1,\frac{aqS^{\frac q2}\rho^{2-q}}{(4-q)|\Omega|^{\frac{6-q}6}}\right\}
\end{align*}
and letting $\lambda\in(0,\lambda_2)$ from now on, we deduce from \eqref{R2.3.1} that $\sup\limits_n\|u_n\|\leq\rho$.
Then, there exists some $u\in H_0^1(\Omega)$ such that, in the sense of subsequence,
\begin{align}\label{R2.3.2}
u_n\rightharpoonup u\ \mbox{in}\ H_0^1(\Omega);\ \ \ \
u_n\rightarrow u\ \mbox{in}\ L^p(\Omega),\ 1\leq p<6;\ \ \ \
u_n(x)\rightarrow u(x)\ \mbox{a.e.}\ \mbox{in}\ \Omega.
\end{align}
Denote $\tilde{u}_n=u_n-u$ and set $\|\tilde{u}_n\|\rightarrow\tilde{\kappa}$, $|\tilde{u}_n|_6\rightarrow\tilde{\nu}$ in the sense of subsequence.
It is easy to see that $0\leq\tilde{\nu}\leq S^{-\frac12}\tilde{\kappa}$.
By \eqref{R2.3.2} and the Br$\acute{\mbox{e}}$zis-Lieb lemma in \cite{W}, we derive
\begin{align}\label{R2.3.3}
0&=\left\langle\mathcal{J}_{b,\lambda}'(u_n),u_n\right\rangle+o(1) \notag\\
&=a\|u\|^2+\frac{b^2S^3}2\|u\|^2+b\|u\|^4-|u|_6^6-\lambda|u|_q^q \notag\\
&\quad+2b\tilde{\kappa}^2\|u\|^2+a\tilde{\kappa}^2+\frac{b^2S^3}2\tilde{\kappa}^2+b\tilde{\kappa}^4-\tilde{\nu}^6 \notag\\
&\quad+\frac{bS}2\left[b^2S^4+4(a+b\tilde{\kappa}^2+b\|u\|^2)S\right]^{\frac12}(\|u\|^2+\tilde{\kappa}^2).
\end{align}
Observe $\left\{|u_n|^4u_n\right\}$ is bounded in $L^{\frac65}(\Omega)$ and $\left\{|u_n|^{q-2}u_n\right\}$ is bounded in $L^{\frac2{q-1}}(\Omega)$, then \eqref{R2.3.2} implies
\begin{align*}
\int_\Omega|u_n|^4u_nudx\rightarrow |u|_6^6\ \ \ \ \mbox{and}\ \ \ \ \int_\Omega|u_n|^{q-2}u_nudx\rightarrow |u|_q^q.
\end{align*}
Then, from \eqref{R2.3.2} and the Br$\acute{\mbox{e}}$zis-Lieb lemma again, we derive
\begin{align}\label{R2.3.4}
0&=\left\langle\mathcal{J}_{b,\lambda}'(u_n),u\right\rangle+o(1) \notag\\
&=a\|u\|^2+\frac{b^2S^3}2\|u\|^2+b\|u\|^4-|u|_6^6-\lambda|u|_q^q \notag\\
&\quad+b\tilde{\kappa}^2\|u\|^2+\frac{bS}2\left[b^2S^4+4(a+b\tilde{\kappa}^2+b\|u\|^2)S\right]^{\frac12}\|u\|^2.
\end{align}
Subtracting \eqref{R2.3.4} from \eqref{R2.3.3} yields
\begin{align*}
S^{-3}\tilde{\kappa}^6-b\tilde{\kappa}^4-\big(a+\frac{b^2S^3}2+b\|u\|^2\big)\tilde{\kappa}^2
-\frac{bS}2\left[b^2S^4+4(a+b\tilde{\kappa}^2+b\|u\|^2)S\right]^{\frac12}\tilde{\kappa}^2\geq0,
\end{align*}
which together with $\tilde{\kappa}\geq0$ implies $\tilde{\kappa}=0$ or
\begin{align}\label{R2.3.5}
\tilde{\kappa}^2\geq \frac{bS^3+S\sqrt{b^2S^4+\big[4a+2b^2S^3+4b\|u\|^2+2bS(b^2S^4+4aS+4b\|u\|^2S)^{\frac12}\big]S}}2.
\end{align}
We claim $\tilde{\kappa}=0$. If not, the weakly lower semicontinuity of norm implies $\|u\|\leq\liminf\limits_n\|u_n\|\leq\rho$,
then it follows from \eqref{R2.3.2}$-$\eqref{R2.3.5} and the Br$\acute{\mbox{e}}$zis-Lieb lemma that
\begin{align*}
\tilde{m}_{b,\lambda}
&=\mathcal{J}_{b,\lambda}(u_n)+o(1) \\
&=\mathcal{J}_{b,\lambda}(u)
+\frac a2\tilde{\kappa}^2+\frac b4\tilde{\kappa}^4+\frac b2\tilde{\kappa}^2\|u\|^2+\frac{b^2S^3}4\tilde{\kappa}^2-\frac16\tilde{\nu}^6 \\
&\quad+\frac1{24}\left[b^2S^4+4(a+b\tilde{\kappa}^2+b\|u\|^2)S\right]^{\frac32}
-\frac1{24}\left[b^2S^4+4(a+b\|u\|^2)S\right]^{\frac32} \\
&\geq\mathcal{J}_{b,\lambda}(u)+\frac a3\tilde{\kappa}^2+\frac b{12}\tilde{\kappa}^4+\frac b3\tilde{\kappa}^2\|u\|^2+\frac{b^2S^3}6\tilde{\kappa}^2 \\
&\quad+\frac{bS}{12}\left[b^2S^4+4(a+b\tilde{\kappa}^2+b\|u\|^2)S\right]^{\frac12}\tilde{\kappa}^2 \\
&> \hat{m}_{b,\lambda},
\end{align*}
which is impossible since $\tilde{m}_{b,\lambda}\leq\hat{m}_{b,\lambda}$.
Then, $\tilde{\kappa}=0$. That is, up to a subsequence, $u_n\rightarrow u$ in~$H_0^1(\Omega)$.
Naturally, $\mathcal{J}_{b,\lambda}(u)=\tilde{m}_{b,\lambda}$ and $\mathcal{J}_{b,\lambda}'(u)=0$.
By the weakly lower semicontinuity of norm again, $\|u\|\leq\rho$.
Hence, $\tilde{m}_{b,\lambda}=\mathcal{J}_{b,\lambda}(u)\geq\hat{m}_{b,\lambda}$,
which and $\tilde{m}_{b,\lambda}\leq\hat{m}_{b,\lambda}$ yield $\tilde{m}_{b,\lambda}=\hat{m}_{b,\lambda}$. This lemma is proved.
\end{prf}

\subsection{Existence of positive solution with positive energy}

In this section, we prove Eq.~$\mathrm{(\ref{K1})}$ has a positive solution with positive energy by a mountain pass theorem originating from the local minimizer $w_{b,\lambda}^1$ and the concentration-compactness principle.~Set
\begin{align*}
\bar{\mathcal{I}}_{b,\lambda}(u)=\frac a2\|u\|^2+\frac b4\|u\|^4-\frac16|u^+|_6^6-\frac\lambda q|u^+|_q^q,\ \ \ \ \forall\ u\in H_0^1(\Omega)
\end{align*}
and
\begin{align*}
\bar{\mathcal{J}}_{b,\lambda}(u)
&=\bar{\mathcal{I}}_{b,\lambda}(u)+\frac1{24}\left[b^2S^4+4\left(a+b\|u\|^2\right)S\right]^{\frac32} \\
&\quad\quad+\frac{b^2S^3}4\|u\|^2-\frac{\left(b^2S^4+4aS\right)^{\frac32}}{24},\ \ \ \ \forall\ u\in H_0^1(\Omega).
\end{align*}
Clearly, $\mathcal{I}_{b,\lambda}(u)=\bar{\mathcal{I}}_{b,\lambda}(u)$ if $u>0$,
and critical point of $\bar{\mathcal{I}}_{b,\lambda}$ corresponds to positive solution of Eq.~\eqref{K1}.

\begin{lem}\label{L2.4}
There exists some $\lambda_3>0$ (independent of $b$) such that, for any $b>0$ and $\lambda\in(0,\lambda_3)$,
the functional $\bar{\mathcal{I}}_{b,\lambda}$ satisfies the $(PS)_c$-condition with $c\in\left(-\infty,c_b+\tilde{m}_{b,\lambda}\right)$, where
\begin{align}\label{R2.4.0}
c_b=\frac{b^3S^6}{24}+\frac{abS^3}{4}+\frac{\left(b^2S^4+4aS\right)^{\frac32}}{24}.
 \end{align}
\end{lem}

\begin{prf}
Let $\{u_n\}\subset H_0^1(\Omega)$  be a $(PS)_c$-sequence for the functional $\bar{\mathcal{I}}_{b,\lambda}$ at $c\in(-\infty,c_b+\tilde{m}_{b,\lambda})$, namely,
\begin{align}\label{R2.4.1}
\bar{\mathcal{I}}_{b,\lambda}(u_n)\xrightarrow{n}c\ \ \ \ \mbox{and}\ \ \ \ \bar{\mathcal{I}}_{b,\lambda}'(u_n)\xrightarrow{n}0\ \mbox{in}\ H_0^{-1}(\Omega).
\end{align}
For any given $\frac16\leq\theta\leq\frac14$, from the H$\ddot{\mbox{o}}$lder and Sobolev inequalities it easily follows that, for $n$ large,
\begin{align*}
c+1+o(\|u_n\|)
&\geq \bar{\mathcal{I}}_{b,\lambda}(u_n)-\theta\left\langle \bar{\mathcal{I}}_{b,\lambda}'(u_n),u_n\right\rangle\\
&=\frac{(1-2\theta)a}2\|u_n\|^2+\frac{(1-4\theta)b}4\|u_n\|^4+\frac{6\theta-1}6|u_n^+|_6^6-\frac{\lambda(1-\theta q)}q|u_n^+|_q^q\\
&\geq\frac{(1-2\theta)a}2\|u_n\|^2-\frac{\lambda(1-\theta q)|\Omega|^{\frac{6-q}6}}{qS^{\frac q2}}\|u_n\|^q,
\end{align*}
which implies that $\{u_n\}$ is bounded in $H_0^1(\Omega)$. Then, there exists some $u\in H_0^1(\Omega)$ such that, in the sense of subsequence,
\begin{align}\label{R2.4.2}
u_n\rightharpoonup u\ \mbox{in}\ H_0^1(\Omega);\ \ \ \
u_n\rightarrow u\ \mbox{in}\ L^p(\Omega),\ 1\leq p<6;\ \ \ \
u_n(x)\rightarrow u(x)\ \mbox{a.e.}\ \mbox{in}\ \Omega.
\end{align}
Since $\left\{(u_n^+)^5\right\}$ is bounded in $L^{\frac65}(\Omega)$ and $\left\{(u_n^+)^{q-1}\right\}$ is bounded in $L^{\frac2{q-1}}(\Omega)$, then \eqref{R2.4.2} implies
\begin{align}\label{R2.4.3}
\int_\Omega(u_n^+)^5udx\rightarrow |u^+|_6^6\ \ \ \ \mbox{and}\ \ \ \ \int_\Omega(u_n^+)^{q-1}udx\rightarrow |u^+|_q^q.
\end{align}
Letting $\bar{u}_n=u_n-u$ and setting $\|\bar{u}_n\|\rightarrow \kappa$, $|u_n^+-u^+|_6\rightarrow \nu$ in the sense of subsequence, then the definition of $S$ implies $\nu\leq S^{-\frac12}\kappa$, we deduce from \eqref{R2.4.1}$-$\eqref{R2.4.3} and the Br$\acute{\mbox{e}}$zis-Lieb lemma in \cite{W}~that
\begin{align}
0=\left\langle \bar{\mathcal{I}}_{b,\lambda}'(u_n),u_n\right\rangle+o(1)
&=a\|u\|^2+b\|u\|^4-|u^+|_6^6-\lambda|u^+|_q^q+a\kappa^2+b\kappa^4+2b\kappa^2\|u\|^2-\nu^6, \label{R2.4.4}\\
0=\left\langle \bar{\mathcal{I}}_{b,\lambda}'(u_n),u\right\rangle+o(1)
&=a\|u\|^2+b\|u\|^4-|u^+|_6^6-\lambda|u^+|_q^q+b\kappa^2\|u\|^2. \label{R2.4.5}
\end{align}
Combining \eqref{R2.4.4} with \eqref{R2.4.5}, we conclude
\begin{align}\label{R2.4.8}
S^{-3}\kappa^6-b\kappa^4-\left(a+b\|u\|^2\right)\kappa^2\geq 0.
\end{align}
It is clear that $\kappa\geq0$. We claim $\kappa=0$. Otherwise, \eqref{R2.4.8} implies
$$
\kappa^2\geq \frac12\left[bS^3+S\sqrt{b^2S^4+4(a+b\|u\|^2)S}\right].
$$
From \eqref{R2.4.1}, \eqref{R2.4.2}, \eqref{R2.4.4}, \eqref{R2.4.5} and the Br$\acute{\mbox{e}}$zis-Lieb lemma, it follows that
\begin{align}\label{R2.4.6}
c
&=\bar{\mathcal{I}}_{b,\lambda}(u_n)+o(1)
=\bar{\mathcal{I}}_{b,\lambda}(u)+\frac{a}2\kappa^2+\frac b4\kappa^4+\frac b2\kappa^2\|u\|^2-\frac16\nu^6 \notag\\
&=\bar{\mathcal{I}}_{b,\lambda}(u)+\frac{a}3\kappa^2+\frac b{12}\kappa^4+\frac b3\kappa^2\|u\|^2.
\end{align}
If $\|u\|\leq\rho$, by \eqref{R2.4.6}, the definition of $\hat{m}_{b,\lambda}$ and Lemma \ref{L2.3+} we derive
\begin{align*}
c
&\geq \bar{\mathcal{I}}_{b,\lambda}(u)
+\frac{a}6\left[bS^3+S\sqrt{b^2S^4+4(a+b\|u\|^2)S}\right] \\
&\quad + \frac b6\|u\|^2\left[bS^3+S\sqrt{b^2S^4+4(a+b\|u\|^2)S}\right] \\
&\quad + \frac b{48}\left[2b^2S^6+4(a+b\|u\|^2)S^3+2bS^4\sqrt{b^2S^4+4(a+b\|u\|^2)S}\right] \\
&=\bar{\mathcal{I}}_{b,\lambda}(u)+\frac{abS^3}4+\frac{b^3S^6}{24}
+\frac1{24}\left[b^2S^4+4\left(a+b\|u\|^2\right)S\right]^{\frac32}+\frac14b^2S^3\|u\|^2 \\
&\geq c_b+\mathcal{J}_{b,\lambda}(u) \\
&\geq c_b+\tilde{m}_{b,\lambda},
\end{align*}
which contradicts with the premise $c<c_b+\tilde{m}_{b,\lambda}$ of this lemma. If $\|u\|>\rho$, due to \eqref{R2.4.5}, there holds
\begin{align*}
\bar{\mathcal{I}}_{b,\lambda}(u)
&\geq\frac a3\|u\|^2+\frac b{12}\|u\|^4-\frac{(6-q)\lambda}{6q}|u|_q^q-\frac{b\kappa^2}6\|u\|^2.
\end{align*}
Take $\lambda_3>0$ such that
\begin{align*}
\lambda_3
\leq \frac{2aqS^{\frac q2}\rho^{2-q}}{(6-q)|\Omega|^{\frac{6-q}6}}
<\frac{6qS^{\frac q2}}{(6-q)|\Omega|^{\frac{6-q}6}}\left[\left(\frac {bS}{12}\sqrt{b^2S^4+4(a+b\rho^2)S}
+\frac a3+\frac{b^2S^3}6\right)\rho^{2-q}+\frac b{12}\rho^{4-q}\right].
\end{align*}
Thereby, we deduce from \eqref{R2.4.6}, the H$\ddot{\mbox{o}}$lder and Sobolev inequalities that, for any $b>0$ and $\lambda\in(0,\lambda_3)$,
\begin{align*}
c&\geq\frac{a}3\kappa^2+\frac b{12}\kappa^4+\frac b6\kappa^2\|u\|^2
+\frac a3\|u\|^2+\frac b{12}\|u\|^4-\frac{(6-q)\lambda}{6q}|u|_q^q \\
&\geq \frac{a}6\left[bS^3+S\sqrt{b^2S^4+4(a+b\|u\|^2)S}\right] \\
&\quad+\frac b{48}\left[2b^2S^6+4(a+b\|u\|^2)S^3+2bS^4\sqrt{b^2S^4+4(a+b\|u\|^2)S}\right] \\
&\quad+\frac b{12}\|u\|^2\left[bS^3+S\sqrt{b^2S^4+4(a+b\|u\|^2)S}\right] \\
&\quad+\frac a3\|u\|^2+\frac b{12}\|u\|^4-\frac{(6-q)\lambda}{6q}|u|_q^q \\
&\geq\frac{abS^3}4+\frac{b^3S^6}{24}+\frac1{24}\left(b^2S^4+4aS\right)\sqrt{b^2S^4+4(a+b\|u\|^2)S} \\
&\quad+\frac {bS}{12}\|u\|^2\sqrt{b^2S^4+4(a+b\|u\|^2)S}
+\frac 16\left(2a+b^2S^3\right)\|u\|^2+\frac b{12}\|u\|^4-\frac{(6-q)|\Omega|^{\frac{6-q}6}\lambda}{6qS^{\frac q2}}\|u\|^q \\
&> c_b+\|u\|^q\left[\left(\frac {bS}{12}\sqrt{b^2S^4+4(a+b\rho^2)S}
+\frac a3+\frac{b^2S^3}6\right)\rho^{2-q}+\frac b{12}\rho^{4-q}-\frac{(6-q)|\Omega|^{\frac{6-q}6}\lambda_3}{6qS^{\frac q2}}\right] \\
&> c_b,
\end{align*}
which is impossible since $c<c_b+\tilde{m}_{b,\lambda}<c_b$. Thus $\kappa=0$, and the proof of this lemma is completed.
\end{prf}

\vspace{6pt}

It is well known that the best Sobolev constant $S$ is attained when $\Omega=\mathbb{R}^3$ by the functions
\begin{align*}
v_\varepsilon(\cdot)=\frac{3^{\frac14}\varepsilon^{\frac12}}{\left(\varepsilon^2+|\cdot|^2\right)^{\frac12}},\ \ \ \forall\ \varepsilon>0,
\end{align*}
which satisfies
\begin{align*}
\int_{\mathbb{R}^3}|\nabla v_\varepsilon|^2dx=\int_{\mathbb{R}^3}|v_\varepsilon|^6dx=S^{\frac32}.
\end{align*}
Without loss of generality, we assume $0\in\Omega$ and $\overline{\mathbb{B}}_{2r}(0)\subset\Omega$ for some $r>0$. Letting $\varphi\in C_0^\infty\left(\Omega,[0,1]\right)$ such that $\varphi(x)\equiv1$ in $\mathbb{B}_r(0)$ while $\varphi(x)\equiv0$ in $\overline{\Omega}\backslash\mathbb{B}_{2r}(0)$ and setting $u_\varepsilon=\varphi v_\varepsilon$, by the similar calculation to \cite{BN,W} we conclude that, as $\varepsilon\rightarrow0^+$,
\begin{align}\label{R8}
\|u_\varepsilon\|^2=S^{\frac32}+O(\varepsilon),\ \ \ \ |u_\varepsilon|_6^6=S^{\frac32}+O(\varepsilon^3)
\end{align}
and
\begin{align}\label{R9}
\left|\int_\Omega\nabla w_{b,\lambda}^1\cdot\nabla u_\varepsilon dx\right|=O(\varepsilon^\frac12).
\end{align}

\vspace{6pt}

\begin{lem}\label{L2.5}
For any $b>0$ and $\lambda\in(0,\lambda_2)$, $\max\limits_{t\geq0}\bar{\mathcal{I}}_{b,\lambda}\big(w_{b,\lambda}^1+tu_\varepsilon\big)<c_b+\tilde{m}_{b,\lambda}$ for $\varepsilon>0$ small enough.
\end{lem}

\begin{prf}
Define for $t\geq0$ the function
\begin{align*}
g(t)=\frac{t^2}2\big(a+b\|w_{b,\lambda}^1\|^2\big)S^{\frac32}+\frac {bt^4}4S^3-\frac{t^6}6S^{\frac32}.
\end{align*}
Through direct calculation, we deduce that $g$ admits a unique maximum point $t_m$ on $[0,+\infty)$, where
\begin{align*}
t_m
=\left[\frac{bS^{\frac32}+\sqrt{b^2S^3+4\big(a+b\|w_{b,\lambda}^1\|^2\big)}}{2}\right]^{\frac12}.
\end{align*}
Then, $g'(t_m)=0$, which together with \eqref{R8} implies that, for any $t\in(0,t_m)\cup(t_m,+\infty)$,
\begin{align}\label{R2.5.1}
g(t)&<g(t_m)=\frac{t_m^2}3\big(a+b\|w_{b,\lambda}^1\|^2\big)S^{\frac32}+\frac{bt_m^4}{12}S^3 \notag\\
&=\frac14\big(a+b\|w_{b,\lambda}^1\|^2\big)bS^3+\frac1{24}b^3S^6
+\frac1{24}\left[b^2S^4+4\left(a+b\|w_{b,\lambda}^1\|^2\right)S\right]^{\frac32} \notag\\
&=\frac{abS^3}4+\frac{b^3S^6}{24}+\frac{(b^2S^4+4aS)^{\frac32}}{24}+\frac{b^2S^3}4\|w_{b,\lambda}^1\|^2 \notag\\
&\quad+\frac1{24}\left[b^2S^4+4\left(a+b\|w_{b,\lambda}^1\|^2\right)S\right]^{\frac32}
-\frac{(b^2S^4+4aS)^{\frac32}}{24}.
\end{align}
Observing $\bar{\mathcal{I}}_{b,\lambda}(w_{b,\lambda}^1)<0$, we deduce $\bar{\mathcal{I}}_{b,\lambda}(w_{b,\lambda}^1+tu_\varepsilon)<0$ for $t>0$ small enough. Moreover, it is easy to show $\lim\limits_{t\rightarrow+\infty}\bar{\mathcal{I}}_{b,\lambda}(w_{b,\lambda}^1+tu_\varepsilon)=-\infty$.
Consequently, for any $\varepsilon>0$, there exists some $t_\varepsilon>0$ such that
\begin{align*}
\bar{\mathcal{I}}_{b,\lambda}(w_{b,\lambda}^1+t_\varepsilon u_\varepsilon)
=\max_{t\geq0} \bar{\mathcal{I}}_{b,\lambda}(w_{b,\lambda}^1+tu_\varepsilon).
\end{align*}
We can certify the following inequalities
\begin{align}\label{R2.5.2}
\begin{cases}
\vspace{2.00mm}
\displaystyle \left(c+d\right)^6\geq c^6+d^6+6c^5d+6cd^5,  \\
\displaystyle \left(c+d\right)^q\geq c^q+q c^{q-1}d, ~q\in(1,2),
\end{cases}
\forall\ c,d>0.
\end{align}
Note $\mathcal{J}_{b,\lambda}'(w_{b,\lambda}^1)=0$ in $H^{-1}_0(\Omega)$, which implies
\begin{align*}
\left(a+b\|w_{b,\lambda}^1\|^2+t_m^2bS^{\frac32}\right)\int_\Omega\nabla w_{b,\lambda}^1\cdot\nabla u_\varepsilon dx
=\int_\Omega (w_{b,\lambda}^1)^5 u_\varepsilon dx+\lambda\int_\Omega (w_{b,\lambda}^1)^{q-1}u_\varepsilon dx,
\end{align*}
then it follows from \eqref{R8}, \eqref{R2.5.1} and \eqref{R2.5.2} that, for $\varepsilon>0$ sufficiently small,
\begin{align}
\bar{\mathcal{I}}_{b,\lambda}(w_{b,\lambda}^1+t_\varepsilon u_\varepsilon)
&=\frac {at_\varepsilon^2}2\|u_\varepsilon\|^2+\frac {bt_\varepsilon^4}4\|u_\varepsilon\|^4
+\frac a2\|w_{b,\lambda}^1\|^2+\frac b4\|w_{b,\lambda}^1\|^4 +\frac {bt_\varepsilon^2}2 \|w_{b,\lambda}^1\|^2\|u_\varepsilon\|^2\notag\\
&\quad+\left(at_\varepsilon+bt_\varepsilon\|w_{b,\lambda}^1\|^2
+bt_\varepsilon^3\|u_\varepsilon\|^2\right)\int_\Omega\nabla w_{b,\lambda}^1\cdot\nabla u_\varepsilon dx
+bt_\varepsilon^2\left(\int_\Omega\nabla w_{b,\lambda}^1\cdot\nabla u_\varepsilon dx\right)^2  \notag\\
&\quad-\frac16\int_\Omega(w_{b,\lambda}^1+t_\varepsilon u_\varepsilon)^6dx
-\frac\lambda q\int_\Omega(w_{b,\lambda}^1+t_\varepsilon u_\varepsilon)^qdx  \label{R2.5.3} \\
&=\frac{t_\varepsilon^2}2\big(a+b\|w_{b,\lambda}^1\|^2\big)\|u_\varepsilon\|^2+\frac {bt_\varepsilon^4}4\|u_\varepsilon\|^4-\frac{t_\varepsilon^6}6|u_\varepsilon|_6^6 \notag\\
&\quad+\frac a2\|w_{b,\lambda}^1\|^2+\frac b4\|w_{b,\lambda}^1\|^4-\frac16|w_{b,\lambda}^1|_6^6-\frac\lambda q|w_{b,\lambda}^1|_q^q \notag\\
&\quad+bt_\varepsilon\left(t_\varepsilon^2\|u_\varepsilon\|^2-t_m^2S^{\frac32}\right)\int_\Omega\nabla w_{b,\lambda}^1\cdot\nabla u_\varepsilon dx
+bt_\varepsilon^2\left(\int_\Omega\nabla w_{b,\lambda}^1\cdot\nabla u_\varepsilon dx\right)^2 \notag\\
&\quad-\frac16\int_\Omega\left(w_{b,\lambda}^1+t_\varepsilon u_\varepsilon\right)^6-(w_{b,\lambda}^1)^6
-(t_\varepsilon u_\varepsilon)^6-6(w_{b,\lambda}^1)^5t_\varepsilon u_\varepsilon dx \notag\\
&\quad-\frac\lambda q\int_\Omega\left(w_{b,\lambda}^1+t_\varepsilon u_\varepsilon\right)^q
-(w_{b,\lambda}^1)^q-q(w_{b,\lambda}^1)^{q-1}t_\varepsilon u_\varepsilon dx \notag\\
&\leq g(t_\varepsilon)+\frac a2\|w_{b,\lambda}^1\|^2+\frac b4\|w_{b,\lambda}^1\|^4-\frac16|w_{b,\lambda}^1|_6^6-\frac\lambda q|w_{b,\lambda}^1|_q^q \notag\\
&\quad +\frac{t_\varepsilon^2}2\big(a+b\|w_{b,\lambda}^1\|^2\big)O(\varepsilon)
+\frac {bt_\varepsilon^4}2O(\varepsilon)-\frac{t_\varepsilon^6}6O(\varepsilon^3) \notag\\
&\quad+\left[bS^{\frac32}t_\varepsilon\left(t_\varepsilon^2-t_m^2\right)
+bO(\varepsilon)t_\varepsilon^3\right]\int_\Omega\nabla w_{b,\lambda}^1\cdot\nabla u_\varepsilon dx \notag\\
&\quad+bt_\varepsilon^2\left(\int_\Omega\nabla w_{b,\lambda}^1\cdot\nabla u_\varepsilon dx\right)^2
-t_\varepsilon^5\int_\Omega w_{b,\lambda}^1 u_\varepsilon^5 dx.
 \label{R2.5.4}
\end{align}
We claim that there exist $T_1,T_2>0$ such that $T_1\leq  t_\varepsilon\leq  T_2$ for all $\varepsilon>0$ small enough.
Indeed, on the one hand, it follows from \eqref{R2.1.1} that $\bar{\mathcal{I}}_{b,\lambda}(w_{b,\lambda}^1+t_\varepsilon u_\varepsilon)\geq \alpha>0$ for all $\varepsilon>0$.
On the other hand, it follows from \eqref{R2.5.3}, \eqref{R8} and the Cauchy-Schwarz inequality that, for sufficiently small $\varepsilon>0$,
\begin{align*}
\bar{\mathcal{I}}_{b,\lambda}(w_{b,\lambda}^1+t_\varepsilon u_\varepsilon)
&\leq \frac {at_\varepsilon^2}2\|u_\varepsilon\|^2+\frac {bt_\varepsilon^4}4\|u_\varepsilon\|^4
+\frac a2\|w_{b,\lambda}^1\|^2+\frac b4\|w_{b,\lambda}^1\|^4 +\frac {3bt_\varepsilon^2}2 \|w_{b,\lambda}^1\|^2\|u_\varepsilon\|^2\notag\\
&\quad+\left(at_\varepsilon+bt_\varepsilon\|w_{b,\lambda}^1\|^2
+bt_\varepsilon^3\|u_\varepsilon\|^2\right)\|w_{b,\lambda}^1\|\|u_\varepsilon\|
-\frac{t_\varepsilon^6}6|u_\varepsilon|_6^6 \notag\\
&\leq \frac a2\|w_{b,\lambda}^1\|^2+\frac b4\|w_{b,\lambda}^1\|^4+
\sqrt{2}\big(a+b\|w_{b,\lambda}^1\|^2\big)\|w_{b,\lambda}^1\|S^{\frac34}t_\varepsilon \notag\\
&\quad+ \big(a+3b\|w_{b,\lambda}^1\|^2\big)S^{\frac32}t_\varepsilon^2  +2\sqrt{2}b\|w_{b,\lambda}^1\|S^{\frac94}t_\varepsilon^3
+bS^3t_\varepsilon^4
-\frac{S^{\frac32}t_\varepsilon^6}{12}.
\end{align*}
By combining the above two sides, we conclude that there exist some constants $T_1,T_2>0$ such that
\begin{align}\label{R2.5.6}
T_1\leq t_\varepsilon\leq T_2\ \ \ \ \mbox{for}\ \varepsilon>0\ \mbox{small\ enough}.
\end{align}
Let $\lim\limits_{\varepsilon\rightarrow0^+}t_\varepsilon=t_*$ in the sense of subsequence.
If $t_*\neq t_m$, by \eqref{R2.5.4}, \eqref{R2.5.6}, \eqref{R9}, \eqref{R2.5.1}, Lemmas \ref{L2.3} and \ref{L2.3+} we deduce, for $\varepsilon>0$ small enough,
\begin{align*}
\bar{\mathcal{I}}_{b,\lambda}(w_{b,\lambda}^1+t_\varepsilon u_\varepsilon)
&\leq g(t_\varepsilon)+\frac a2\|w_{b,\lambda}^1\|^2+\frac b4\|w_{b,\lambda}^1\|^4-\frac16|w_{b,\lambda}^1|_6^6
-\frac\lambda q|w_{b,\lambda}^1|_q^q+O(\varepsilon)  \notag\\
&\quad+bS^{\frac32}T_2\left(T_2^2+t_m^2\right)\left|\int_\Omega\nabla w_{b,\lambda}^1\cdot\nabla u_\varepsilon dx\right|
+bT_2^2\left|\int_\Omega\nabla w_{b,\lambda}^1\cdot\nabla u_\varepsilon dx\right|^2 \notag\\
&\leq g(t_\varepsilon)+\frac a2\|w_{b,\lambda}^1\|^2+\frac b4\|w_{b,\lambda}^1\|^4-\frac16|w_{b,\lambda}^1|_6^6
-\frac\lambda q|w_{b,\lambda}^1|_q^q+O(\varepsilon^{\frac12}) \\
&<g(t_m)+\frac a2\|w_{b,\lambda}^1\|^2+\frac b4\|w_{b,\lambda}^1\|^4-\frac16|w_{b,\lambda}^1|_6^6
-\frac\lambda q|w_{b,\lambda}^1|_q^q \\
&=c_b+\tilde{m}_{b,\lambda}.
\end{align*}
If $t_*=t_m$, note $\min\limits_{\overline{\mathbb{B}}_{2r}(0)}w_{b,\lambda}^1>0$ since $w_{b,\lambda}^1(\cdot)>0$ in $\Omega$ and $w_{b,\lambda}^1\in C(\Omega)$,
then there holds
\begin{align*}
\lim_{\varepsilon\rightarrow0^+}\varepsilon^{-\frac12}
\left[O(\varepsilon)+bS^{\frac32}T_2\left|t_\varepsilon^2-t_m^2\right|O(\varepsilon^{\frac12})
-\varepsilon^{\frac12}T_1^5\min_{\overline{\mathbb{B}}_{2r}(0)}w_{b,\lambda}^1
\int_{\mathbb{B}_1(0)} \frac{3^{\frac54}}{(1+|x|^2)^{\frac52}} dx\right]<0.
\end{align*}
In view of this, by \eqref{R2.5.4}, \eqref{R9}, \eqref{R2.5.1}, Lemmas \ref{L2.3} and \ref{L2.3+} we deduce
\begin{align*}
\bar{\mathcal{I}}_{b,\lambda}(w_{b,\lambda}^1+t_\varepsilon u_\varepsilon)
&\leq g(t_m)+\frac a2\|w_{b,\lambda}^1\|^2+\frac b4\|w_{b,\lambda}^1\|^4-\frac16|w_{b,\lambda}^1|_6^6-\frac\lambda q|w_{b,\lambda}^1|_q^q \\
&\quad+O(\varepsilon)+bS^{\frac32}T_2\left|t_\varepsilon^2-t_m^2\right|O(\varepsilon^{\frac12})
-T_1^5\int_\Omega w_{b,\lambda}^1 u_\varepsilon^5 dx \\
&\leq\frac{abS^3}4+\frac{b^3S^6}{24}+\frac{(b^2S^4+4aS)^{\frac32}}{24}+\mathcal{J}_{b,\lambda}(w_{b,\lambda}^1) \\
&\quad+O(\varepsilon)+bS^{\frac32}T_2\left|t_\varepsilon^2-t_m^2\right|O(\varepsilon^{\frac12})
-\varepsilon^{\frac12}T_1^5\min_{\overline{\mathbb{B}}_{2r}(0)}w_{b,\lambda}^1
\int_{\mathbb{B}_1(0)} \frac{3^{\frac54}}{(1+|x|^2)^{\frac52}} dx \\
&<c_b+\tilde{m}_{b,\lambda}
\end{align*}
for $\varepsilon>0$ sufficiently small. Based on the above discussion, we complete the proof of this lemma.
\end{prf}

\begin{thm}\label{T2.6}
Let $\lambda_4=\min\{\lambda_2,\lambda_3\}$, then Eq.~\eqref{K1} has a positive solution $u_{b,\lambda}^2\in H_0^1(\Omega)$ satisfying $\mathcal{I}_{b,\lambda}(u_{b,\lambda}^2)>0$ for any $b>0$ and $\lambda\in(0,\lambda_4)$.
\end{thm}

\begin{prf}
By Lemma \ref{L2.1}, there exist some $\rho,\alpha>0$ such that, for any $b>0$ and $\lambda\in(0,\lambda_4)$,
\begin{align}\label{R2.6.1}
\inf\limits_{u\in S_\rho}\bar{\mathcal{I}}_{b,\lambda}(u)
\geq\inf\limits_{u\in S_\rho}\mathcal{I}_{b,\lambda}(u)
\geq \alpha,
\end{align}
where $S_\rho=\left\{u\in H_0^1(\Omega):\|u\|=\rho\right\}$.
Define the minimax class
\begin{align*}
\Gamma_{b,\lambda}=\left\{\gamma\in C\big([0,1],H_0^1(\Omega)\big):\gamma(0)=w_{b,\lambda}^1,\ \|\gamma(1)\|>\rho\ \mbox{and}\ \bar{\mathcal{I}}_{b,\lambda}(\gamma(1))<0\right\}
\end{align*}
and set the corresponding minimax level
\begin{align*}
c_{b,\lambda}=\inf_{\gamma\in\Gamma_{b,\lambda}}\max_{t\in[0,1]}\bar{\mathcal{I}}_{b,\lambda}(\gamma(t)).
\end{align*}
Thanks to \cite[Theorem 2.8]{W}, there exists a sequence $\{u_n\}\subset H_0^1(\Omega)$ such that
$\bar{\mathcal{I}}_{b,\lambda}(u_n)\xrightarrow{n}c_{b,\lambda}$ and
$\bar{\mathcal{I}}_{b,\lambda}'(u_n)\xrightarrow{n}0$ in $H_0^{-1}(\Omega)$.
By \eqref{R2.6.1}, we may deduce $c_{b,\lambda}\geq\alpha>0$.
Moreover, Lemma \ref{L2.5} implies
\begin{align*}
c_{b,\lambda}\leq\max\limits_{t\geq0}\bar{\mathcal{I}}_{b,\lambda}(w_{b,\lambda}^1+tu_\varepsilon)<c_b+\tilde{m}_{b,\lambda}.
\end{align*}
Then, due to Lemma \ref{L2.4}, there exists some $u_{b,\lambda}^2$ such that $u_n\rightarrow u_{b,\lambda}^2$ in $H_0^1(\Omega)$ up to a  subsequence, which implies $\bar{\mathcal{I}}_{b,\lambda}'(u_{b,\lambda}^2)=0$ and $\bar{\mathcal{I}}_{b,\lambda}(u_{b,\lambda}^2)=c_{b,\lambda}$.
Clearly, $u_{b,\lambda}^2\neq0$ and $\langle\bar{\mathcal{I}}_{b,\lambda}'(u_{b,\lambda}^2),(u_{b,\lambda}^2)^-\rangle=0$ implies $(u_{b,\lambda}^2)^-=0$.
Thus, by the strong maximum principle, we derive $u_{b,\lambda}^2(\cdot)>0$. The proof is finished.
\end{prf}

\vspace{8pt}
Now, with the above preliminaries in hand, we end the proof of our main existence result.
\vspace{8pt}

\begin{prf1}
For any $b>0$ and $\lambda\in(0,\lambda_4)$, combining Theorems \ref{T2.2} and \ref{T2.6}, we obtain two positive solutions $u_{b,\lambda}^1$ and $u_{b,\lambda}^2$ of Eq.~$\mathrm{(\ref{K1})}$ satisfying
$$
\mathcal{I}_{b,\lambda}(u_{b,\lambda}^1)=m_{b,\lambda}<0<\mathcal{I}_{b,\lambda}(u_{b,\lambda}^2).
$$
Next, we prove that $u_{b,\lambda}^1$ is indeed a positive ground state solution of Eq.~$\mathrm{(\ref{K1})}$ for $\lambda>0$ small enough. For this aim, it suffices to verify, for $\lambda>0$ small enough,
$$
m_{b,\lambda}=\overline{m}_{b,\lambda}:=\inf\limits_{\mathcal{N}_{b,\lambda}}\mathcal{I}_{b,\lambda},
$$
where
\begin{align*}
\mathcal{N}_{b,\lambda}=\left\{u\in H_0^1(\Omega)\backslash\{0\}: \mathcal{I}_{b,\lambda}'(u)=0\ \mbox{in}\ H_0^{-1}(\Omega)\right\}.
\end{align*}
On the one hand, since $u_{b,\lambda}^1\in\mathcal{N}_{b,\lambda}$, then $\overline{m}_{b,\lambda}\leq \mathcal{I}_{b,\lambda}(u_{b,\lambda}^1)=m_{b,\lambda}<0$.
On the other hand, choose a sequence $\{u_n\}\subset \mathcal{N}_{b,\lambda}$ such that
\begin{align*}
\mathcal{I}_{b,\lambda}(u_n)\xrightarrow{n}\overline{m}_{b,\lambda}\ \ \ \ \mbox{and}\ \ \ \ \mathcal{I}_{b,\lambda}'(u_n)=0\ \mbox{in}\ H_0^{-1}(\Omega).
\end{align*}
It follows from the H$\ddot{\mbox{o}}$lder and Sobolev inequalities that
\begin{align}\label{P1.1+}
\overline{m}_{b,\lambda}+o(1)
&=\mathcal{I}_{b,\lambda}(u_n)-\frac14\left\langle\mathcal{I}_{b,\lambda}'(u_n),u_n\right\rangle \notag\\
&=\frac a4\|u_n\|^2+\frac1{12}|u_n|_6^6-\frac{\lambda(4-q)}{4q}|u_n|_q^q \notag\\
&\geq \frac a4\|u_n\|^2-\frac{\lambda(4-q)|\Omega|^{\frac{6-q}6}}{4qS^{\frac q2}}\|u_n\|^q.
\end{align}
Take
\begin{align*}
\Lambda=\min\left\{\lambda_4,\frac{aqS^{\frac q2}\rho^{2-q}}{(4-q)|\Omega|^{\frac{6-q}6}}\right\}
\end{align*}
and let $\lambda\in(0,\Lambda)$ from now on.
Observing $\overline{m}_{b,\lambda}<0$, we may deduce from \eqref{P1.1+} that $\sup\limits_n\|u_n\|\leq\rho$.
Then, there exists some $u\in H_0^1(\Omega)$ such that, in the sense of subsequence,
\begin{align}\label{P1.2+}
u_n\rightharpoonup u\ \mbox{in}\ H_0^1(\Omega);\ \ \ \
u_n\rightarrow u\ \mbox{in}\ L^p(\Omega),\ 1\leq p<6;\ \ \ \
u_n(x)\rightarrow u(x)\ \mbox{a.e.}\ \mbox{in}\ \Omega.
\end{align}
Denote $\bar{u}_n=u_n-u$ and set $\|\bar{u}_n\|\rightarrow\bar{\kappa}$, $|\bar{u}_n|_6\rightarrow\bar{\nu}$ up to a subsequence.
It is easy to show $0\leq\bar{\nu}\leq S^{-\frac12}\bar{\kappa}$.
By \eqref{P1.2+} and the Br$\acute{\mbox{e}}$zis-Lieb lemma, we deduce
\begin{align}\label{P1.3+}
0&=\left\langle\mathcal{I}_{b,\lambda}'(u_n),u_n\right\rangle+o(1) \notag\\
&=a\|u\|^2+b\|u\|^4-|u|_6^6-\lambda|u|_q^q \notag\\
&\quad+a\bar{\kappa}^2+2b\bar{\kappa}^2\|u\|^2+b\bar{\kappa}^4-\bar{\nu}^6
\end{align}
and
\begin{align}\label{P1.4+}
0=\left\langle\mathcal{I}_{b,\lambda}'(u_n),u\right\rangle +o(1)
=a\|u\|^2+b\|u\|^4-|u|_6^6-\lambda|u|_q^q+b\bar{\kappa}^2\|u\|^2.
\end{align}
Subtracting \eqref{P1.4+} from \eqref{P1.3+} yields
\begin{align*}
0&=\bar{\nu}^6-b\bar{\kappa}^4-\left(a+b\|u\|^2\right)\bar{\kappa}^2 \\
&\leq S^{-3}\bar{\kappa}^6-b\bar{\kappa}^4-\left(a+b\|u\|^2\right)\bar{\kappa}^2,
\end{align*}
which together with $\bar{\kappa}\geq0$ implies $\bar{\kappa}=0$ or
\begin{align}\label{P1.5+}
\bar{\kappa}^2\geq \frac{bS^3+S\sqrt{b^2S^4+4(a+b\|u\|^2)S}}2.
\end{align}
We claim $\bar{\kappa}=0$. Otherwise, the weakly lower semicontinuity of norm implies $\|u\|\leq\liminf\limits_n\|u_n\|\leq\rho$,
then we deduce from \eqref{P1.2+}$-$\eqref{P1.5+}, the Br$\acute{\mbox{e}}$zis-Lieb lemma and the definition of $m_{b,\lambda}$ that
\begin{align*}
\overline{m}_{b,\lambda}
&=\mathcal{I}_{b,\lambda}(u_n)+o(1) \\
&=\mathcal{I}_{b,\lambda}(u)
+\frac 12\left(a+b\|u\|^2\right)\bar{\kappa}^2+\frac b4\bar{\kappa}^4-\frac16\bar{\nu}^6 \\
&\geq\mathcal{I}_{b,\lambda}(u)+\frac 13\left(a+b\|u\|^2\right)\bar{\kappa}^2+\frac b{12}\bar{\kappa}^4 \\
&> m_{b,\lambda},
\end{align*}
which is impossible since $\overline{m}_{b,\lambda}\leq m_{b,\lambda}$.
Hence, $\bar{\kappa}=0$. That is, up to a subsequence, $u_n\rightarrow u$ in~$H_0^1(\Omega)$.
Naturally, $\mathcal{I}_{b,\lambda}(u)=\overline{m}_{b,\lambda}$ and $\mathcal{I}_{b,\lambda}'(u)=0$ in $H_0^{-1}(\Omega)$.
By $\|u\|\leq\rho$, there holds $\overline{m}_{b,\lambda}=\mathcal{I}_{b,\lambda}(u)\geq m_{b,\lambda}$.
Now, combining the above two hands, we derive $\overline{m}_{b,\lambda}=m_{b,\lambda}$ for any $\lambda\in(0,\Lambda)$.
Thus $u_{b,\lambda}^1$ is a positive ground state solution of Eq.~$\mathrm{(\ref{K1})}$ once $\lambda\in(0,\Lambda)$.
The proof of Theorem \ref{T1} is completed.
\end{prf1}

\section{Proof of Theorems \ref{T2} and \ref{T3}}\label{Se2}

With Theorem \ref{T1} in hand, we further study the asymptotic behaviour of positive solutions as $b\rightarrow0^+$ and $\lambda\rightarrow0^+$ in order.
For Eq.~\eqref{SC1}, the natural functional $\mathcal{I}_{0,\lambda}\in C^1(H_0^1(\Omega),\mathbb{R})$ is defined by
\begin{align*}
\mathcal{I}_{0,\lambda}(u)=\frac a2\|u\|^2-\frac16|u|_6^6-\frac\lambda q|u|_q^q,\ \ \ \ \forall\ u\in H_0^1(\Omega).
\end{align*}

\begin{prf2}
Due to the proof of Theorem \ref{T1}, we conclude that $\{u_{b_n,\lambda}^1\}$ and $\{u_{b_n,\lambda}^2\}$~satisfy
\begin{align}\label{P1.3.1}
\mathcal{I}_{b_n,\lambda}(u_{b_n,\lambda}^1)=\overline{m}_{b_n,\lambda}<0
<\alpha\leq\mathcal{I}_{b_n,\lambda}(u_{b_n,\lambda}^2)=c_{b_n,\lambda}<c_{b_n}+\hat{m}_{b_n,\lambda}.
\end{align}
Take $u\in H_0^1(\Omega)$ such that $|u|_q=1$ and set
$$
u_\lambda=\Big(\frac{2-q}2\Big)^{\frac1q}\left[\frac{(6-q)|\Omega|^{\frac{6-q}6}}{4(aS)^{\frac q2}}\right]^{\frac2{q(2-q)}}\lambda^{\frac1{2-q}}u.
$$
It easily follows from \eqref{R2.1.2} that
\begin{align}\label{P1.3.6}
\limsup_{n\to\infty}\hat{m}_{b_n,\lambda}
&\leq\limsup_{n\to\infty}\lim\limits_{t\rightarrow0^+}\frac{\mathcal{I}_{b_n,\lambda}(tu_\lambda)}{t^q} \notag\\
&\leq \limsup_{n\to\infty} \lim\limits_{t\rightarrow0^+}
\left[\frac {at^{2-q}}2\|u_\lambda\|^2+\frac {b_nt^{4-q}}4\|u_\lambda\|^4-\frac{t^{6-q}}6|u_\lambda|_6^6-\frac{\lambda} q|u_\lambda|_q^q\right] \notag\\
&=-\frac\lambda q|u_\lambda|_q^q
<-\frac{2-q}{3q}\left[\frac{\lambda(6-q)|\Omega|^{\frac{6-q}6}}{4(aS)^{\frac q2}}\right]^{\frac2{2-q}},
\end{align}
which together with \eqref{R2.4.0} implies
\begin{align}\label{P1.3.2}
\limsup_{n\to\infty}\left(c_{b_n}+\hat{m}_{b_n,\lambda}\right)<\frac13(aS)^{\frac32}
-\frac{2-q}{3q}\left[\frac{\lambda(6-q)|\Omega|^{\frac{6-q}6}}{4(aS)^{\frac q2}}\right]^{\frac2{2-q}}.
\end{align}
Observe $\mathcal{I}_{b_n,\lambda}'(u_{b_n,\lambda}^i)=0$, where $i=1,2$ and we acknowledge this notation from now on,
by \eqref{P1.3.2}, the H$\ddot{\mbox{o}}$lder and Sobolev inequalities, there results
\begin{align*}
\frac13(aS)^{\frac32}&>\limsup_{n\to\infty}\Big[\mathcal{I}_{b_n,\lambda}(u_{b_n,\lambda}^i)
-\frac14\left\langle\mathcal{I}_{b_n,\lambda}'(u_{b_n,\lambda}^i),u_{b_n,\lambda}^i\right\rangle\Big] \\
&>\frac a4\|u_{b_n,\lambda}^i\|^2+\frac1{12}|u_{b_n,\lambda}^i|_6^6-\frac{\lambda(4-q)}{4q}|u_{b_n,\lambda}^i|_q^q\\
&\geq\frac a4\|u_{b_n,\lambda}^i\|^2-\frac{\lambda(4-q)|\Omega|^{\frac{6-q}6}}{4qS^{\frac q2}}\|u_{b_n,\lambda}^i\|^q.
\end{align*}
In view of this fact, $\{u_{b_n,\lambda}^i\}$ is bounded.
Then, there exists $u_\lambda^i \in H_0^1(\Omega)$ such that, up to a subsequence,
\begin{align}\label{P1.3.3}
\begin{cases}
\displaystyle u_{b_n,\lambda}^i \overset{n}\rightharpoonup u_\lambda^i\ \mbox{in}\ H_0^1(\Omega);\\
\displaystyle u_{b_n,\lambda}^i\xrightarrow{n}u_\lambda^i\ \mbox{in}\ L^p(\Omega),\ 1\leq p<6;\\
\displaystyle u_{b_n,\lambda}^i(x)\xrightarrow{n}u_\lambda^i(x)\ \mbox{a.e.}\ \mbox{in}\ \Omega.
\end{cases}
\end{align}
Since \eqref{P1.3.3} implies that $(u_{b_n,\lambda}^i)^5\overset{n}\rightharpoonup (u_\lambda^i)^5$ in $L^{\frac65}(\Omega)$ and $(u_{b_n,\lambda}^i)^{q-1}\overset{n}\rightharpoonup (u_\lambda^i)^{q-1}$ in $L^{\frac2{q-1}}(\Omega)$ in the sense of subsequence, from \eqref{P1.3.3} we may deduce
\begin{align}\label{P1.3.4}
0 = \left\langle\mathcal{I}_{b_n,\lambda}'(u_{b_n,\lambda}^i), u_\lambda^i\right\rangle+o(1)
= a\|u_\lambda^i\|^2-|u_\lambda^i|_6^6-\lambda|u_\lambda^i|_q^q.
\end{align}
Set $v_n^i=u_{b_n,\lambda}^i-u_\lambda^i$, $\limsup\limits_n\|v_n^i\|=\kappa_i$ and $\limsup\limits_n|v_n^i|_6=\nu_i$, due to \eqref{P1.3.3} and the Br$\acute{\mbox{e}}$zis-Lieb lemma,
\begin{align}\label{P1.3.5}
0 = \left\langle\mathcal{I}_{b_n,\lambda}'(u_{b_n,\lambda}^i), u_{b_n,\lambda}^i\right\rangle+o(1)
= a\|u_\lambda^i\|^2-|u_\lambda^i|_6^6-\lambda|u_\lambda^i|_q^q+ a\kappa_i^2-\nu_i^6.
\end{align}
Subtracting \eqref{P1.3.4} from \eqref{P1.3.5}, we derive $a\kappa_i^2=\nu_i^6$, which and the Sobolev inequality imply $a\kappa_i^2\leq S^{-3}\kappa_i^6$.
Consequently, $\kappa_i=0$ or $\kappa_i^2\geq(aS^3)^{\frac12}$. We assert $\kappa_i=0$. If not, it follows from \eqref{P1.3.1}$-$\eqref{P1.3.5}, the Br$\acute{\mbox{e}}$zis-Lieb lemma, the H$\ddot{\mbox{o}}$lder and Sobolev inequalities that
\begin{align*}
&\frac13(aS)^{\frac32}-\frac{2-q}{3q}\left[\frac{\lambda(6-q)|\Omega|^{\frac{6-q}6}}{4(aS)^{\frac q2}}\right]^{\frac2{2-q}} \\
&\quad\quad > \limsup_{n\to\infty} \mathcal{I}_{b_n,\lambda}(u_{b_n,\lambda}^i)
=\frac a3\|u_\lambda^i\|^2-\frac{\lambda(6-q)}{6q}|u_\lambda^i|_q^q+\frac a3\kappa_i^2 \\
&\quad\quad\geq\frac a3\|u_\lambda^i\|^2-\frac{\lambda(6-q)|\Omega|^{\frac{6-q}6}}{6qS^{\frac q2}}\|u_\lambda^i\|^q+\frac a3\kappa_i^2 \\
&\quad\quad\geq\min_{t\geq0}\left[\frac a3t^2-\frac{\lambda(6-q)|\Omega|^{\frac{6-q}6}}{6qS^{\frac q2}}t^q\right]+\frac13(aS)^{\frac32} \\
&\quad\quad= \frac13(aS)^{\frac32}-\frac{2-q}{3q}\left[\frac{\lambda(6-q)|\Omega|^{\frac{6-q}6}}{4(aS)^{\frac q2}}\right]^{\frac2{2-q}},
\end{align*}
a contradiction. Naturally, $\kappa_i=0$. That is, $u_{b_n,\lambda}^i\xrightarrow{n}u_\lambda^i$ in $H_0^1(\Omega)$ in the sense of subsequence.
Then, by \eqref{P1.3.1} and \eqref{P1.3.3} we may deduce $\mathcal{I}_{0,\lambda}'(u_\lambda^1)=\mathcal{I}_{0,\lambda}'(u_\lambda^2)=0$ and
\begin{align*}
&\mathcal{I}_{0,\lambda}(u_\lambda^1)=\limsup_{n\to\infty}\overline{m}_{b_n,\lambda}<0, \\
&\mathcal{I}_{0,\lambda}(u_\lambda^2)=\limsup_{n\to\infty}\mathcal{I}_{b_n,\lambda}(u_{b_n,\lambda}^2)\geq\alpha>0,
\end{align*}
which implies $u_\lambda^1\neq0$ and $u_\lambda^2\neq0$.
Moreover, recalling $u_\lambda^1(\cdot)\geq0$ and $u_\lambda^2(\cdot)\geq0$ in $\Omega$,
we deduce from the strong maximum principle that $u_\lambda^1>0$ and $u_\lambda^2>0$.
Thus, $u_\lambda^1$ and $u_\lambda^2$ are positive solutions of Eq.~\eqref{SC1}, where $u_\lambda^1$ admits negative energy while $u_\lambda^2$ admits positive energy. The proof is completed.
\end{prf2}

\vspace{8pt}

\begin{prf3}
According to the proof of Theorem \ref{T1}, we know that $u_{b,\lambda_n}^1(\cdot)>0$ satisfies
$$
\mathcal{I}_{b,\lambda_n}'(u_{b,\lambda_n}^1)=0\ \  \mbox{and}\ \
\mathcal{I}_{b,\lambda_n}(u_{b,\lambda_n}^1)=\overline{m}_{b,\lambda_n}\ \ \ \mbox{for\ all}\ n.
$$
We deduce from the H$\ddot{\mbox{o}}$lder and Sobolev inequalities that
\begin{align*}
0&>\overline{m}_{b,\lambda_n}
\\&=\mathcal{I}_{b,\lambda_n}(u_{b,\lambda_n}^1)-\frac14\left\langle\mathcal{I}_{b,\lambda_n}'(u_{b,\lambda_n}^1),u_{b,\lambda_n}^1\right\rangle \\
&\geq \frac a4\|u_{b,\lambda_n}^1\|^2+\frac1{12}|u_{b,\lambda_n}^1|_6^6-\frac{\lambda_n(4-q)}{4q}|u_{b,\lambda_n}^1|_q^q \\
&\geq \frac a4\|u_{b,\lambda_n}^1\|^2-\frac{\lambda_n(4-q)|\Omega|^{\frac{6-q}6}}{4qS^{\frac q2}}\|u_{b,\lambda_n}^1\|^q,
\end{align*}
which implies $\{\|u_{b,\lambda_n}^1\|\}$ is bounded and then $\|u_{b,\lambda_n}^1\|\xrightarrow{n}0$.

By the proof of Theorem \ref{T2.6}, there~holds
\begin{align}\label{R11}
\mathcal{I}_{b,\lambda_n}(u_{b,\lambda_n}^2)=c_{b,\lambda_n}\geq\alpha>0 \ \ \ \mbox{for\ all}\ n.
\end{align}
Using the similar elliptic estimates to the proof of Lemma \ref{L2.3}, we derive $u_{b,\lambda_n}^2\in C^{1,\alpha}(\overline{\Omega})$.
Consequently, $u_{b,\lambda_n}^2\in L^\infty(\Omega)$. We claim $|u_{b,\lambda_n}^2|_\infty\xrightarrow{n}+\infty$.
Otherwise, $\sup_n|u_{b,\lambda_n}^2|_\infty<+\infty$. Then it follows from the $L^p$-theory, Morrey's embedding theorem and the Schauder estimate that $\{u_{b,\lambda_n}^2\}$ is bounded in~$C^{2,\alpha}(\overline{\Omega})$.
Moreover, thanks to the Ascoli-Arzel$\grave{\mbox{a}}$ theorem, there exists some $u_b$ such that $u_{b,\lambda_n}^2\xrightarrow{n}u_b$ in $C^2(\overline{\Omega})$ in the sense of subsequence. Naturally, $u_b$ satisfies
\begin{align}\label{P1.5.1}
\begin{cases}
\displaystyle -\Big(a+b\int_\Omega|\nabla u_b|^2dx\Big)\Delta u_b=|u_b|^4u_b  &\mbox{in}\ \Omega, \\
\displaystyle u_b=0 &\mbox{on}\ \partial\Omega,
\end{cases}
\end{align}
which together with the result \cite[Theorem 1.5: Eq.~\eqref{P1.5.1} has no nontrivial solution if $\Omega$ is strictly~star-shaped]{N} implies $u_b=0$.
As a result, $u_{b,\lambda_n}^2\xrightarrow{n}0$ in $H_0^1(\Omega)$ and then $\mathcal{I}_{b,\lambda_n}(u_{b,\lambda_n}^2)\xrightarrow{n}0$,
which contradicts with \eqref{R11}. That is, our claim $|u_{b,\lambda_n}^2|_\infty\xrightarrow{n}+\infty$ is true.
Thus the proof of Theorem \ref{T3} is completed.
\end{prf3}

\vspace{0.4cm}
\noindent{\bf Data availability}~Data sharing is not applicable to this article as no datasets were generated or analysed during the current study.

\vspace{0.4cm}
\noindent{\bf Conflict of interest}~The authors declare that there is no conflict of interest.


\begin{thebibliography}{30}
\setlength{\itemsep}{-0.20pt}
\bibitem{ABC} A. Ambrosetti, H. Br$\acute{\mbox{e}}$zis, G. Cerami, Combined effects of concave and convex nonlinearities in some elliptic problems, J. Funct. Anal. 122 (1994) 519--477.
\bibitem{BN} H. Br$\acute{\mbox{e}}$zis, L. Nirenberg, Positive solutions of nonlinear elliptic equations involving critical Sobolev exponents, Commun. Pure Appl. Math. 36 (1983) 437--477.
\bibitem{CXW} X. Cao, J. Xu, J. Wang, Multiple positive solutions for Kirchhoff type problems involving concave and convex nonlinearities in $\mathbb{R}^3$, Electron. J. Differential Equations 2016 (2016) 16 pp.
\bibitem{CKW} C.-Y. Chen, Y.-C. Kuo, T.-F. Wu, The Nehari manifold for a Kirchhoff type problem involving sign-changing weight functions, J. Differential Equations. 250 (2011) 1876--1908.
\bibitem{CT} S.T. Chen, X.H. Tang, Normalized solutions for Kirchhoff equations with Sobolev critical exponent and mixed nonlinearities, Math. Ann. 391 (2025) 2783--2836.
\bibitem{CWL} B.T. Cheng, X. Wu, J. Liu, Multiple solutions for a class of Kirchhoff type problems with concave nonlinearity, Nonlinear Differential Equations Appl. 19 (2012) 521--537.
\bibitem{K} G. Kirchhoff, Mechanik, Teubner, Leipzig, 1883.
\bibitem{LLG} C.-Y. Lei, G.-S. Liu, L.-T. Guo, Multiple positive solutions for a Kirchhoff type problem with a critical nonlinearity, Nonlinear Anal. Real World Appl. 31 (2016) 343--355.
\bibitem{LPKT} J.-F. Liao, Y. Pu, X.-F. Ke, C.-L. Tang, Multiple positive solutions for kirchhoff type problems involving concave-convex nonlinearities, Commun. Pure Appl. Anal. 16 (2017) 2157--2175.
\bibitem{L} J.L. Lions, On some questions in boundary value problems of mathematical physics, In: Contemporary developments in continuum mechanics and partial differential equations (Proc. Internat. Sympos., Inst. Mat., Univ. Fed. Rio de Janeiro, Rio de Janeiro), North-Holland Math. Stud. 30 (1977) 284--346.
\bibitem{N} D. Naimen, Positive solutions of Kirchhoff type elliptic equations involving a critical Sobolev exponent, NoDEA Nonlinear Differential Equations Appl. 21 (2014) 885--914.
\bibitem{P} S.I. Poho$\check{\mbox{z}}$aev, Eigenfunctions of the equation $\Delta u+\lambda f(u)=0$, Soviet Math. Dokl. 6 (1965) 1408--1411 (translated from the Russian Dokl. Akad. Nauk SSSR, 165 (1965) 36--39).
\bibitem{SL} Y.J. Sun, X. Liu, Existence of positive solutions for Kirchhoff type problems with critical exponent, J. Partial Differ. Equ. 25 (2012) 187--198.
\bibitem{W} M. Willem, Minimax Theorems. Birkh$\ddot{\mbox a}$user, Basel (1996).
\bibitem{XWT} Q.-L. Xie, X.-P. Wu, C.-L. Tang, Existence and multiplicity of solutions for Kirchhoff type problem with critical exponent, Commun. Pure Appl. Anal. 12 (2013) 2773--2786.
\bibitem{ZL} C.H. Zhang, Z.S. Liu, Multiplicity of nontrivial solutions for a critical degenerate Kirchhoff type problem, Appl. Math. Lett. 69 (2017) 87--93
\bibitem{ZT} X.-J. Zhong, C.-L. Tang, Multiple positive solutions to a Kirchhoff type problem involving a critical nonlinearity, Comput. Math. Appl. 72 (2016) 2865--2877.
\bibitem{ZF} X. Zhu, H. Fan, The sign-changing solutions for a class of Kirchhoff-type problems with critical Sobolev exponents in bounded domains, Z. Angew. Math. Phys. 75 (2024) 29 pp.

\end{thebibliography}
\end{document}